\documentclass[12pt,a4paper]{amsart}
\usepackage{latexsym}
\usepackage{amssymb}
\usepackage{xcolor}

\usepackage[bookmarks=true,hyperindex,pdftex,colorlinks,citecolor=blue]{hyperref}

 \newtheorem{theo}{Theorem}[section]
 \newtheorem{definition}[theo]{Definition}
 
  \newtheorem{lem}[theo]{Lemma}

\newtheorem{propos}[theo]{Proposition}

\newcommand{\R}{{\mathbb{R}}}

\newcommand{\loglike}[1]{\mathop{\rm #1}\nolimits}
\newcommand{\ex}{\loglike{ext}}

\newcommand{\spn}{\mathrm{span}}

\newcommand{\bea}{\begin{eqnarray*}}
\newcommand{\eea}{\end{eqnarray*}}
\newcommand{\beq}{\begin{eqnarray}}
\newcommand{\eeq}{\end{eqnarray}}

\newcommand{\supp}{{\mathrm{supp}}}

\numberwithin{equation}{section}
\begin{document}

\title[Non-expansive bijections to the unit ball of $\ell_1$-sum]{Non-expansive bijections to the unit ball of $\ell_1$-sum of  strictly convex Banach spaces}

\author[Kadets]{V. Kadets}

\address[Kadets]{ School of Mathematics and Informatics, V.N. Karazin Kharkiv  National University, 61022 Kharkiv, Ukraine
\newline  \href{http://orcid.org/0000-0002-5606-2679}{ORCID: \texttt{0000-0002-5606-2679}}}
\email{v.kateds@karazin.ua }

\author[Zavarzina]{O. Zavarzina}
\address[Zavarzina]{ School of Mathematics and Informatics, V.N. Karazin Kharkiv  National University, 61022 Kharkiv, Ukraine
\newline  \href{http://orcid.org/0000-0002-5731-6343}{ORCID: \texttt{0000-0002-5731-6343}}}
\email{olesia.zavarzina@yahoo.com}


\thanks{The research of the first author is done in frames of Ukrainian Ministry of Science and Education Research Program 0115U000481}
\maketitle

\begin{abstract}
Extending recent results by Cascales, Kadets, Orihuela and Wingler (2016), Kadets and Zavarzina (2017), and Zavarzina (2017) we demonstrate that for every  Banach space $X$ and every collection $Z_i, i\in I$ of strictly convex Banach spaces every non-expansive bijection from the unit ball of $X$ to the unit ball of sum of $Z_i$ by $\ell_1$  is an isometry.
\end{abstract}


\section{Introduction}
This article is motivated by the challenging open problem, posed by B.~Cascales, V.~Kadets, J.~Orihuela and E.J.~Wingler in 2016 \cite{CKOW2016}, whether it is true that for every Banach space $X$ its unit ball $B_X$ is \emph{Expand-Contract plastic}, in other words, whether it is true that every non-expansive bijective automorphism of $B_X$ is an isometry. It looks surprising that such a general property, if true, remained unnoticed during the long history of Banach space theory development.  On the other hand, if there is a counterexample, it is not an easy task to find it, because of known partial positive results. Namely, in finite-dimensional case the Expand-Contract plasticity of $B_X$ follows from compactness argument: it is known  \cite{NaiPioWing} that every totally bounded metric space is Expand-Contract plastic. For infinite-dimensional case, the main result of \cite{CKOW2016} ensures Expand-Contract plasticity of the unit ball of every strictly convex Banach space, in particular of Hilbert spaces and of all $L_p$ with $1 < p < \infty$. An example of not strictly convex infinite-dimensional space with the same property of the ball is presented in \cite[Theorem 1] {KZ}. This example is $\ell_1$ or, more generally, $\ell_1(\Gamma)$, where the same proof needs just minor modifications.

In this paper we ``mix'' results from \cite[Theorem 2.6] {CKOW2016} and \cite[Theorem 1] {KZ} and demonstrate  the Expand-Contract plasticity of the ball of $\ell_1$-sum of an arbitrary collection of strictly convex spaces. Moreover, we demonstrate a stronger result: for every  Banach space $X$ and every collection $Z_i, i\in I$ of strictly convex Banach spaces we prove that every non-expansive bijection from the unit ball of $X$ to the unit ball of $\ell_1$-sum of spaces $Z_i$ is an isometry. Analogous results for non-expansive bijections acting from the unit ball of an arbitrary Banach space to unit balls of finite-dimensional or strictly convex spaces, as well as to the unit ball  of $\ell_1$ were established recently in \cite{Zav}.

Our demonstration uses several ideas from preceding papers mentioned above, but elaborates them substantially in order to overcome the difficulties that appear on the way in this new, more general situation.

\section{Notations and auxiliary statements}
Before proving the corresponding theorem we will give the notations and results which we need in our exposition.

In this paper we deal with real Banach spaces. As usual, for a Banach space $E$ we denote by  $S_E$ and $B_E$ the unit sphere and the closed unit ball of $E$ respectively. A map $F\colon U \to V$ between metric spaces $U$ and $V$ is called \emph{non-expansive}, if $\rho(F(u_1), F(u_2)) \le \rho(u_1, u_2)$ for all $u_1, u_2 \in U$, so in the case of non-expansive map $F\colon B_X \to B_Z$ considered below we have $\|F(x_1) - F(x_2)\| \le \|x_1 - x_2\|$ for $x_1, x_2 \in B_X$.

 For a convex set $M \subset E$ we denote by $\ex(M)$ the set of \emph{extreme points} of $M$. Recall that $z \in \ex(M)$ if for every non-trivial line segment $[u, v]$ containing  $z$ in its interior, at least one of the endpoints $u, v$ should not belong to $M$. Recall also that a space $E$ is called \emph{strictly convex} when  $S_E = \ex(B_E)$. In strictly convex spaces the triangle inequality is strict for all pairs of vectors with different directions. That is, for every $e_1,e_2\in E$ such that $e_1\neq ke_2$, $k\in (0, +\infty)$, $\|e_1 +e_2\| < \|e_1\| + \|e_2\|$.

Let $I$ be an index set, and $Z_i, i\in I$ be a fixed collection of strictly convex Banach spaces. We consider the sum of $Z_i$ by $\ell_1$ and denote it by $Z$. According to the definition, this means that $Z$ is the set of all points $z = (z_i)_{i\in I}$, where $z_i \in Z_i, i\in I$ with  at most countable support  $\supp (z) : =\{i: z_i\neq 0\}$ and such that $\sum_{i\in I}\|z_i\|_{Z_i}< \infty$. The space $Z$ is equipped with the natural norm
\beq \label{eq:ell1gam-norm}
||z||=\|(z_i)_{i\in I}\|=\sum_{i\in I}\|z_i\|_{Z_i}.
\eeq
 Remark, that even if $I$ is uncountable, the corresponding sum in \eqref{eq:ell1gam-norm} reduces to an ordinary at most countable sum $\sum_{i\in \supp (z)}\|z_i\|_{Z_i}$, which does not depend on the order of its terms, so there is no need to introduce an ordering on $I$ and to appeal to any kind of definition for uncountable sum, when we speak about our space $Z$.

 In the sequel we will regard each $Z_i$ as a subspace of $Z$ in the following natural way:  $Z_i = \{z \in Z \colon \supp(z) \subset \{i\}\}$.  It is well-known and easy to check that in this notation
 $$
 \ex(B_Z)= \bigcup_{ i\in I} S_{Z_i}.
 $$
Remark also that under this notation each  $z \in Z$ can be written in a unique way as a sum  $z = \sum_{i\in I} z_i$,  $z_i \in Z_i$ with at most countable number of non-zero terms, and the series converges absolutely.

  \begin{definition}
 Let $E$ be a Banach space and  $H \subset E$ be a subspace. We will say that a linear projector $P\colon E \to H$ is strict if $\|P\|=1$ and for any $x\in E\setminus H$ we have $\|P(x)\|<\|x\|$.
 \end{definition}
 \begin{lem}\label{strictprojector}
 Every strict projector $P\colon E \to H$ possesses the following property:
  for every $ x \in E \setminus H$ and every $y \in H$ we have  $\|P(x-y)\|$ $ < \|x-y\|$.
 \end{lem}
 \begin{proof}
 If $x\notin H$ then $x-y\notin H$, and since projector $P$ is strict we get
$\|P(x-y)\|  < \|x-y\|.$
 \end{proof}

Consider a finite subset $J \subset I$ and an arbitrary collection $z = (z_i)_{i \in J}$,  $z_i \in S_{Z_i}, i \in J$. For each of these  $z_i$ pick a supporting functional  $z_i^*\in S_{{Z_i}^*}$, i.e. such a norm-one functional that  $z_i^*(z_i)=1$. The strict convexity of $Z_i$ implies that $z_i^*(x) < 1$ for all $x\in B_{Z_i} \setminus \{z_i\}$, $i \in J$.  Denote $z^* = (z_i^*)_{i \in J}$ and define the map  $P_{z, z^*} \colon Z\to  \spn\{z_i,  i\in J\}$,
$$ P_{z, z^*}((y_i)_{i\in I}) = \sum_{i\in J} z_i^*(y_i)z_i.$$

\begin{lem} \label{projector}
The map $P_{z, z^*}$ is a strict projector onto $\spn\{z_i,  i\in J\}$.
\end{lem}
\begin{proof}
According to definition, we have to check that
\begin{enumerate}
\item $P_{z, z^*}$ is a projector on $\spn\{z_i,  i\in J\}$.
\item $\|P_{z, z^*}\| = 1$.
\item  If $(y_i)_{i\in I}\notin \spn\{z_i,  i\in J\}$ then $\|P_{z, z^*}((y_i)_{i\in I})\|  < \|(y_i)_{i\in I}\|$.
\end{enumerate}
\noindent \textbf{Demonstration of (1).} This is true since
\begin{align*}
P_{z, z^*}^2((y_i)_{i\in I})&=P_{z, z^*}\left(\sum_{i\in J} z_i^*(y_i)z_i\right) = \sum_{i\in J}  z_i^*\left( z_i^*(y_i)z_i \right) z_i \\
&= \sum_{i\in J} z_i^*(y_i) z_i^*( z_i ) z_i = \sum_{i\in J} z_i^*(y_i)z_i = P_{z, z^*}((y_i)_{i\in I}).
\end{align*}

\noindent \textbf{Demonstration of (2).}  One may write
\begin{align} \label{align_1}
\nonumber \quad \quad  \|P_{z, z^*}((y_i)_{i\in I})\| &= \left\|\sum_{i\in J} z_i^*(y_i)z_i\right\|=\sum_{i\in J}|z_i^*(y_i)|  \\
&\leq \sum_{i\in J}\|y_i\| \leq \sum_{i\in I}\|y_i\|  = \|(y_i)_{i\in I}\|.
\end{align}

\noindent \textbf{Demonstration of (3).} If  there is $N \in I \setminus J$ such that $y_N \neq 0$ the item is obvious by the second line in \eqref{align_1}.  If $y_N=0$ for all $N \in I \setminus J$  then since $y = \sum_{i\in J}y_i \notin  \spn\{z_i,  i\in J\}$ there is a $j \in J$ such that $y_j \notin \spn\{z_j\}$ and consequently $|z_j^*(y_j)|< \|y_j\|$ for this $j$. Thus, the inequality \eqref{align_1}  becomes strict when we pass from its first line to the second one.
\end{proof}

\begin{propos}[Brower's invariance of domain principle \cite{Brouwer1912}] \label{Brower}
Let $U$ be an open subset of $\R^n$ and $f : U \to \R^n$ be an injective continuous map, then $f(U)$ is open in $\R^n$.
\end{propos}
\begin{propos}[{\cite[Proposition 4]{KZ}}] \label{prop-surject}
Let $X$ be a  finite-dimensional normed space and $V$ be a subset of $B_X$ with the following two properties:  $V$ is homeomorphic to $B_X$ and $V \supset S_X$. Then $V=B_X$.
\end{propos}
\begin{propos}[P. Mankiewicz \cite{mank}] \label{Mankiewicz}
If $X, Y$ are real Banach spaces,  $A\subset X$ and $B \subset Y$ are convex with non-empty interior, then every bijective isometry $F : A \to B$ can be extended to a bijective affine isometry $\tilde F : X \to Y $.
\end{propos}
\begin{propos}[Extracted from {\cite[Theorem 2.3]{CKOW2016}} and {\cite[Theorem 2.1]{Zav}}] \label{lem-old-old}
Let $F: B_X \to B_Y$ be a non-expansive bijection.  Then
\begin{enumerate}
\item  $F(0) = 0$.
\item  $F^{-1}(S_Y) \subset S_X$.
\item  If $F(x)$ is an extreme point of $B_Y$, then $F(ax) = a F(x)$ for all $a \in (-1,1)$.
\end{enumerate}
\end{propos}

\begin{lem}[{\cite[Lemma 2.3]{Zav}}] \label{conv-smooth-prel}
  Let $X, Y$ be Banach spaces, $F\colon B_X \to B_Y$ be a bijective non-expansive map such that $F(S_X) = S_Y$. Let $V \subset S_X$ be such a subset that $F(av) = a F(v)$ for all $a \in [-1,1]$, $v \in V$.  Denote $A = \{tx: x \in V, t \in [-1,1] \}$, then $F|_A$ is a bijective  isometry between $A$ and $F(A)$.
\end{lem}
\begin{lem}\label{preim-of-sph}
 Let $X, Y$ be real Banach spaces,  $F: B_X \to B_Y$ be a bijective non-expansive map such that for
every $v \in F^{-1}(S_Y)$ and every $t \in [-1,1]$ the condition
$F(tv) = t F(v)$ holds true. Then $F$ is an isometry.
\end{lem}
\begin{proof}
According to Proposition \ref{lem-old-old} $F(0) = 0$ and $F^{-1}(S_Y) \subset S_X$. Let us first show  that  $F(S_X) \subset S_Y$, that is $F(S_X) = S_Y$.

For arbitrary $x\in S_X$ consider the point $y = \frac{F(x)}{\|F(x)\|}\in S_Y$ and define $\hat{x} = F^{-1}(y)$. Then, denoting $t = \|F(x)\|$ we get
$$
F(x) = ty = tF(\hat{x}) = F(t\hat{x}).
$$
By injectivity, this implies $x = t \hat{x}$. Since $\|\hat{x}\|=1=\|x\|$, we have that $ \|F(x)\| = t =1$, that is $F(x) \in S_Y$.

Now we may apply Lemma \ref{conv-smooth-prel} to $V = F^{-1}(S_Y) = S_X$ and   $A = \{tx: x \in S_X, t \in [-1,1] \} = B_X$.  Then $F(A) = B_Y$,  so  Lemma \ref{conv-smooth-prel} says that $F$ is an isometry.
\end{proof}

\section{Main result}

 \begin{theo}\label{teo:X+Y}
Let $X$ be a Banach space, $Z_i, i\in I$ be a fixed collection of strictly convex Banach spaces, $Z$ be the  $\ell_1$-sum of the collection  $Z_i, i\in I$, and $F\colon B_X \to B_Z$ be a non-expansive bijection. Then $F$ is an isometry.
\end{theo}

The essence of the proof consists in Lemma \ref{lem-old-claim} below which analyzes the behavior of $F$ on some typical finite-dimensional parts of the ball.

Under the conditions of Theorem \ref{teo:X+Y} consider a finite subset $J \subset I$, $|J| = n$ and pick collections $z = (z_i)_{i \in J}$,  $z_i \in S_{Z_i}, i \in J$, $z^* = (z_i^*)_{i \in J}$, where each  $z_i^*\in S_{{Z_i}^*}$ is a supporting functional for the corresponding $z_i$. Denote $ x_i = F^{-1}(z_i) \in S_X$.
Denote by $U_n$ and $\partial U_n$ the unit ball and the unit sphere of $\spn\{x_i\}_{i\in  J}$ respectively. Let $V_n$ and $\partial V_n$ be the unit ball and the unit sphere of $\spn\{z_i\}_{i \in J}$.

\begin{lem}\label{lem-old-claim}
For every  collection $(a_i)_{i \in  J}$ of reals with $\sum_{i\in  J}{a_i x_i}\in U_n$
\begin{equation} \label{eq*}
\left\|\sum_{i\in  J}{a_i x_i}\right\| = \sum_{i\in  J}|a_i|,
\end{equation}
(which means in particular that $U_n$ isometric to the unit ball of $n$-dimensional $\ell_1$), and
\begin{equation} \label{*}
F\left(\sum_{i\in  J}{a_i x_i}\right) = \sum_{i\in  J}{a_i z_i}.
\end{equation}
\end{lem}
\begin{proof}
We will use the induction in $n$. Recall,  that $z_i \in \ex B_Z$.
This means that for $n = 1$, our Lemma follows from item (3)  of  Proposition \ref{lem-old-old}. Now assume the validity of  Lemma for index sets of $n-1$ elements, and let us prove it for $|J| = n$. Fix an $m \in J$ and denote $J_{n-1} = J \setminus \{m\}$,
At first, let us prove that
\begin{equation} \label{eq1}
F(U_n) \subset V_n.
\end{equation}
To this end, consider $r \in U_n$. If $r$ is of the form $a_m x_m$ the statement follows from (3)  of  Proposition \ref{lem-old-old}. So we must consider $r = \sum_{i\in  J}{a_i x_i}$, $\sum_{i\in  J}{|a_i|}\leq 1$ with $\sum_{i\in  J_{n-1}}|a_i| \neq 0$. Denote the expansion of $F(r)$ by $F(r) = (v_i)_{i\in I}$. For the element
$$
r_1 = \sum_{i\in  J_{n-1}}{\frac{ a_i}{\sum_{j\in  J_{n-1}}|a_j|}x_i}
$$
by the induction hypothesis
$$
F(r_1) =   \sum_{i\in  J_{n-1}}\frac{a_i}{\sum_{j\in  J_{n-1}}|a_j|}  z_i.
$$
Moreover, on the one hand,
$$
\left\|\sum_{i\in  J}{a_i x_i}\right\|\leq \sum_{i\in  J}{|a_i|}.
$$
On the other hand,
\begin{align*}
\left\|\sum_{i\in  J}{a_i x_i}\right\|&=\left\|\sum_{i\in  J_{n-1}}{a_i x_i}- (-a_m x_m)\right\|\geq \left\|F\left(\sum_{i\in  J_{n-1}}{a_i x_i}\right)- F(-a_m x_m)\right\| \\
&=\left\|\sum_{i\in J_{n-1}}a_i z_i- a_m z_m\right\|=\sum_{i\in  J}{|a_i|}.
\end{align*}
Thus, \eqref{eq*} is demonstrated and we may write the following inequalities:
\begin{align*}
2 &=  \|F(r_1) - \frac{a_m}{|a_m|}z_m\|  \leq \left\|F(r_1) - \sum_{i\in J} v_i \right\|+\left\|\sum_{i\in J} v_i -F\left(\frac{a_m}{|a_m|}x_m\right)\right\| \\
&= \|F(r_1) - F(r)\|+\left\|F(r)-F\left(\frac{a_m}{|a_m|}x_m\right)\right\|- 2\left\|\sum_{i\in I \setminus J} v_i\right\|  \\
&\leq \|F(r_1) - F(r)\|+\left\|F(r)-F\left(\frac{a_m}{|a_m|}x_m\right)\right\|\\
&\leq \left\| \sum_{i\in  J_{n-1}}{\frac{a_i}{\sum_{j\in  J_{n-1}}|a_j|} x_i} - \sum_{i\in  J}{a_i x_i}\right\|+\left\|\sum_{i\in  J}{a_i x_i} - \frac{a_m}{|a_m|} x_m \right\|  \\ &\leq \sum_{i\in  J_{n-1}}\left| a_i -  \frac{a_i}{ \sum_{j\in  J_{n-1}}|a_j|} \right| + |a_m| + \sum_{i \in  J_{n-1}}| a_i| + \left|a_m - \frac{a_m}{|a_m|}\right|  \\
&= \sum_{i\in  J_{n-1}}| a_i| \left(1 + \left| 1 -  \frac{1}{ \sum_{j\in  J_{n-1}}|a_j|} \right| \right) +|a_m|\left(1 + \left| 1 -  \frac{1}{|a_m|} \right| \right) = 2.
\end{align*}
So, all the inequalities in this chain are in fact equalities, which implies that
$$
F(r) = \sum_{i\in J}v_i \, \textrm{ and } \, \|F(r_1) - F(r)\|+\left\|F(r)-F\left(\frac{a_m}{|a_m|}x_m\right)\right\|=2.
$$

 Remind that our goal is to check that $F(r)\in V_n$. Suppose by contradiction that $F(r)=\sum_{i\in J} v_i \notin V_n$ and denote for reader's convenience by $s=\sum_{j\in J_{n-1}}|z_j^*(v_j)|$.
  Then in notations of Lemma \ref{projector}
 \begin{align*}
2&= \left\|F\left(\sum_{i\in J_{n-1}}\frac{z_i^*(v_i)}{s}{x_i}\right) - F(r)\right\|+\left\|F(r)-F\left(\frac{z_m^*(v_m)}{|z_m^*(v_m)|}x_m\right)\right\| \\
 &=\left\|\sum_{i\in J_{n-1}}\left(\frac{z_i^*(v_i)}{s}z_i-v_i\right)-v_m\right\|+\left\|\sum_{i\in J_{n-1}} v_i + v_m-\frac{z_m^*(v_m)}{|z_m^*(v_m)|}z_m\right\| \\
  &> \left\|P_{z, z^*}\left(\sum_{i\in J_{n-1}}\left(\frac{z_i^*(v_i)}{s}z_i-v_i\right)-v_m \right)\right\| \\
  &+\left\|P_{z, z^*}\left(\sum_{i\in J_{n-1}} v_i + v_m - \frac{z_m^*(v_m)}{|z_m^*(v_m)|}z_m\right)\right\|  \\
  &=\left\|\sum_{i\in J_{n-1}}\left(\frac{z_i^*(v_i)}{s}-z_i^*(v_i)z_i\right)- z_m^*(v_m)z_m\right\| \\
  &+ \left\|\sum_{i\in J_{n-1}}z_i^*(v_i)z_i + x_{m}^*(v_m)-\frac{z_m^*(v_m)}{|z_m^*(v_m)|}z_m \right\| = \sum_{i\in J_{n-1}}\left| z_i^*(v_i) -  \frac{z_i^*(v_i)}{s} \right|  \\ &+ |z_m^*(v_m)| + \sum_{i\in J_{n-1}}| z_i^*(v_i)| + \left|z_m^*(v_m) - \frac{z_m^*(v_m)}{|z_m^*(v_m)|}\right| \\
 &=\sum_{i\in J_{n-1}}| z_i^*(v_i)| \left(1 + \left| 1 -  \frac{1}{s} \right| \right) + |z_m^*(v_m)|\left(1 + \left| 1 -  \frac{1}{|z_m^*(v_m)|} \right| \right) = 2.
 \end{align*}
Observe, that we have written the strict inequality in this chain because of Lemmas \ref{projector} and \ref{strictprojector}. The above contradiction means that our assumption was wrong, that is
 \begin{equation}\label{eq1}
  F(U_n)\subset V_n.
 \end{equation}
 Further we are going to prove the inclusion
 \begin{equation}\label{eq2}
   \partial V_n \subset F(U_n).
 \end{equation}
We will argue by contradiction. Let there is a point $\sum_{i\in J} t_i\in \partial V_n \setminus F(U_n)$ and denote $ \tau = F^{-1}(\sum_{i\in J} t_i)$. Then $||\sum_{i\in J} t_i|| = 1$ and $\tau \notin U_N$. Rewrite
$$
\sum_{i\in J} t_i = \sum_{i\in J}\|t_i\|\hat{t_i}, \quad \hat{t_i}\in S_{Z_i}.
$$
Pick some supporting functionals ${t_i}^*$ in the points $\hat{t_i}$, $i\in J$ and denote  $t=(\hat{t_i})_{i\in J}$ and $t^*=({t_i}^*)_{i\in J}$.
Let us  demonstrate that $F(\alpha \tau) \in V_n$  for all $\alpha\in[0,1]$. Indeed, if $F(\alpha \tau) \notin V_n$  for some $\alpha$,   denoting $F(\alpha \tau) =  \sum_{i\in I}w_i$,  we deduce from Lemmas \ref{projector} and \ref{strictprojector} the following contradiction
\begin{align*}
1&=\|0-\alpha \tau\|+\|\alpha \tau-\tau\|\geq\left\|0-\sum_{i\in I}w_i\right\|+\left\|\sum_{i\in I}w_i-\sum_{i\in J} t_i\right\| \\
&=2\left\|\sum_{i\in I \setminus J}w_i\right\|+ \left\|\sum_{i\in J}w_i\right\|+\left\|\sum_{i\in J}w_i -\sum_{i\in J}t_i\right\| \\
&>\left\|P_{t,t^*}\left(\sum_{i\in J}w_i\right)\right\|+\left\|P_{t,t^*}\left(\sum_{i\in J}w_i\right)-\sum_{i\in J}t_i\right\|\\
&=\left\|\sum_{i\in J} t_i^*(w_i)\hat{t_i}\right\|+\left\|\sum_{i\in J} t_i^*(w_i)\hat{t_i}-\sum_{i\in J}t_i\right\| \\
&=\sum_{i\in J}|t_i^*(w_i)|+\sum_{i\in J}\left|\|t_i\|-t_i^*(w_i)\right|\geq \sum_{i\in J}\|t_i\| = 1.
\end{align*}
 Note that $F(U_n)$ contains a relative neighborhood of 0 in $V_n$ (here we use item (1) of Proposition \ref{lem-old-old}  and Proposition \ref{Brower}),  so the continuous curve $\{F(\alpha \tau) \colon \alpha \in [0,1]\}$  connecting $0$ with $\sum_{i\in J}t_i$  in  $V_n$  has a non-trivial intersection with  $F(U_n)$.  This implies that there is a $a\in [0,1]$ such that $F(a\tau) \in F(U_n)$. Since $a\tau \notin U_n$ this contradicts the injectivity of $F$. Inclusion \eqref{eq2} is proved.
Now, inclusions  \eqref{eq1} and \eqref{eq2} together with Lemma \ref{prop-surject} imply $F(U_n) = V_n$. Observe, that $U_n$ and $V_n$ are isometric to the unit ball of $n$-dimensional $\ell_1$, so they can be considered as two copies of the same compact metric space. Hence Expand-Contract plasticity of totally bounded metric spaces \cite{NaiPioWing} implies that every bijective non-expansive map from $U_n$ onto $V_n$  is an isometry. In particular, $F$ maps $U_n$ onto $V_n$ isometrically. Finally, the application of Lemma \ref{Mankiewicz} gives us that the restriction of $F$ to $U_n$ extends to a linear map from $\spn\{x_i, i\in J\}$ to $\spn\{z_i,$ $i\in J\}$, which evidently implies \eqref{*}.
\end{proof}

\begin{proof}[Proof of Theorem \ref{teo:X+Y}]
 Our aim is to apply Lemma \ref{preim-of-sph}. To satisfy the conditions of the lemma, for every $z\in S_Z$ we must regard $y = F^{-1}(z)$ and check that for every $t\in [-1,1]$
 \beq \label{eq: Fty}
 F(ty) = tz.
 \eeq
 To this end let us denote $J_z = \supp(z)$, and write
 $$
 z =\sum_{i\in J_z} z_i = \sum_{i\in J_z} \|z_i\|\tilde{z}_i,
 $$
 where $\tilde{z}_i \in S_{Z_i}$. Let us also denote for all $i\in J_{z}$
$$
 x_i := F^{-1}(\tilde{z}_i) \in S_X.
$$
For $J_{z}$ being finite formula  \eqref{*} of Lemma \ref{lem-old-claim} implies that
$$
y = F^{-1}(z) =  F^{-1}\left(\sum_{i\in J_z} \|z_i\|\tilde{z}_i\right) = \sum_{i\in J_z} \|z_i\|x_i, \quad \textrm{and}
$$
$$
 F(ty) = F\left(\sum_{i\in J_z} t \|z_i\|x_i\right) = \sum_{i\in J_z} t \|z_i\|\tilde{z}_i = tz,
$$
which demonstrates \eqref{eq: Fty} in this case. It remains to demonstrate \eqref{eq: Fty} for the case of countable  $J_z$. In this case we can write  $J_z = \{i_1, i_2, \ldots\}$ and consider its finite subsets $J_n = \{i_1, i_2, \ldots, i_n\}$. For these finite subsets $\sum_{i\in J_n} \|z_i\| \le 1$, so $\sum_{i\in J_n} \|z_i\|x_i \in U_n := B_{\spn\{x_i\}_{i\in  J_n}}$, and  we may deduce from Lemma \ref{lem-old-claim} that
$$
F\left(\sum_{i\in J_n} \|z_i\|x_i\right) = \sum_{i\in J_n} \|z_i\|\tilde{z}_i.
$$
Passing to limit as $n \to \infty$ we get
$$
F\left(\sum_{i\in J_z} \|z_i\|x_i\right) = \sum_{i\in J_z} \|z_i\|\tilde{z}_i = z, \, \textrm{ i.e. } y = F^{-1}(z) = \sum_{i\in J_z} \|z_i\|x_i.
$$
One more application of formula \eqref{*} of Lemma \ref{lem-old-claim} gives us
$$
F\left(\sum_{i\in J_n} t \|z_i\|x_i\right) = \sum_{i\in J_n} t \|z_i\|\tilde{z}_i,
$$
which after passing to limit ensures \eqref{eq: Fty}:
$$
 F(ty) = F\left(\lim_{n \to \infty}\sum_{i\in J_n} t \|z_i\|x_i\right) = \lim_{n \to \infty} \sum_{i\in J_n} t \|z_i\|\tilde{z}_i  = \sum_{i\in J_z} t \|z_i\|\tilde{z}_i = tz.
$$
This fact demonstrates applicability of Lemma \ref{preim-of-sph} to our $F$ and thus completes the proof of the theorem.
\end{proof}

\bibliographystyle{amsplain}

\begin{thebibliography}{10}

\bibitem{Brouwer1912} Brouwer L.E.J.   \textit{Beweis der Invarianz des $n$-dimensionalen Gebiets}, Mathematische Annalen, {\bf 71} (1912), 305--315.
\bibitem{CKOW2016} Cascales B.,  Kadets V., Orihuela J.,  Wingler E.J. \textit{Plasticity of the unit ball of a strictly convex Banach space}, Revista de la Real Academia de Ciencias Exactas, F\'{\i}sicas y Naturales. Serie A. Matem\'aticas,  {\bf 110(2)}(2016), 723--727.
\bibitem{KZ} Kadets V.,  Zavarzina O.  \textit{Plasticity of the unit ball of $\ell_1$},  Visn. Hark. nac. univ. im. V.N. Karazina, Ser.: Mat. prikl. mat. meh.,  \textbf{83} (2017), 4--9.
\bibitem{mank}  Mankiewicz P. \textit{On extension of isometries in normed linear spaces}, Bull. Acad. Polon. Sci., S\'er. Sci. Math. Astronom. Phys., {\bf 20} (1972),   367--371.
\bibitem{NaiPioWing}  Naimpally S. A., Piotrowski Z., Wingler E. J.  \textit{Plasticity in metric spaces}, J. Math. Anal. Appl., \textbf{313} (2006), 38--48.
\bibitem{Zav} Zavarzina O. \textit{Non-expansive bijections between unit balls of Banach spaces},   \href{https://arxiv.org/abs/1704.06961v2}{arXiv:1704.06961v2}, to appear in Annals of Functional Analysis.
\end{thebibliography}

\end{document}